\documentclass[]{amsart}

\usepackage{amssymb}  
\usepackage{latexsym} 
\usepackage{comment}
\usepackage{url}

\DeclareFontEncoding{OT2}{}{} 
\newcommand{\textcyr}[1]{%
 {\fontencoding{OT2}\fontfamily{wncyr}\fontseries{m}\fontshape{n}\selectfont #1}}

\usepackage[all]{xy}

\newcommand{\Sha}{{\mbox{\textcyr{Sh}}}}

\newcommand{\Z}{{\mathbb Z}}
\newcommand{\Q}{{\mathbb Q}}

\newcommand{\PP}{{\mathbb P}}

\newcommand{\kbar}{{\overline{k}}}

\newcommand{\calO}{{\mathcal O}}

\newcommand{\To}{\longrightarrow}

\DeclareMathOperator{\End}{End}

\DeclareMathOperator{\Hom}{Hom}

\DeclareMathOperator{\Gal}{Gal}

\DeclareMathOperator{\Pic}{Pic}

\DeclareMathOperator{\ALB}{\bf Alb}
\DeclareMathOperator{\Jac}{Jac}

\DeclareMathOperator{\HH}{H}

\newcommand{\gal}[1]{\mathfrak{g}_{#1}}

   {\begin{list}{}{\settowidth{\labelwidth}{#1}%
                   \setlength{\leftmargin}{\labelwidth}%
                   \addtolength{\leftmargin}{\labelsep}}}%
   {\end{list}}

\newtheorem{Theorem}{Theorem}
\newtheorem{Lemma}[Theorem]{Lemma}

\theoremstyle{definition}

\newtheorem{Remark}[Theorem]{Remark}

\begin{document}

\title{Locally trivial torsors that are not Weil-Ch\^atelet divisible}

\author{Brendan Creutz}
\address{School of Mathematics and Statistics, University of Sydney, NSW 2006, Australia}
\email{brendan.creutz@sydney.edu.au}

\begin{abstract}
For every prime $p$ we give infinitely many examples of torsors under abelian varieties over $\Q$ that are locally trivial but not divisible by $p$ in the Weil-Ch\^atelet group. We also give an example of a locally trivial torsor under an elliptic curve over $\Q$ which is not divisible by $4$ in the Weil-Ch\^atelet group. This gives a negative answer to a question of Cassels.
\end{abstract}

\maketitle
\section{Introduction}
Let $A$ be an abelian variety over a number field $k$ with algebraic closure $\kbar$. A $k$-torsor under $A$ is a variety $T/k$ together with a simply transitive algebraic group action of $A$ on $T$ which is defined over $k$. The isomorphism classes of such torsors are parameterized by the Galois cohomology group $\HH^1(A) := \HH^1(k,A(\kbar))$, commonly known as the {\em Weil-Ch\^atelet group}. The subgroup $\Sha(A) \subset \HH^1(A)$ consisting of classes which become trivial over every completion of $k$ is known as the {\em Tate-Shafarevich group}. A torsor is trivial precisely when it possesses a rational point, so the classes in $\Sha(A)$ are precisely those represented by torsors which have points everywhere locally.

It is conjectured that $\Sha(A)$ is finite, and this is known for some modular abelian varieties including all elliptic curves over $\Q$ of analytic rank at most one. For principally polarized abelian varieties it is known that the order of $\Sha(A)$, if finite, is either a square or $2$ times a square \cite{PoonenStoll}. The proof of this fact is that there is a canonical antisymmetric pairing on $\Sha(A)$ whose kernel is the maximal divisible subgroup. The pairing was first defined by Cassels \cite{CasselsIII} in the case of elliptic curves. While doing so he raised the question of whether the elements of $\Sha(A)$ are infinitely divisible in the larger group $\HH^1(A)$ \cite[Problem 1.3]{CasselsIII}, noting that an affirmative answer would allow him to show that the kernel of the pairing is the maximal divisible subgroup. He was subsequently able to use the weaker result, furnished by Tate, that $\Sha(A) \subset p\HH^1(A)$ when $A$ is an elliptic curve and $p$ is prime to obtain his result on the pairing \cite{CasselsIV}. He notes, however, that the question of divisibility remained open \cite[Problem (b)]{CasselsIV}. 

The issue was then taken up by Bashmakov \cite{Bashmakov64,Bashmakov72} who showed that for certain CM abelian varieties the elements of $\Sha(A)$ are infinitely $p$-divisible in $\HH^1(A)$ whenever $p$ is sufficiently large (depending on $A$). This would also hold, of course, if $\Sha(A)$ were finite. Bashmakov's method makes use of results of Serre on the representation of Galois in the Tate module \cite{Serre1}. These have subsequently been extended and improved by Bogomolov and by Serre \cite{Bogomolov, Serre2}, in effect extending Bashmakov's result to arbitrary abelian varieties. Recently \c Ciperiani and Stix \cite{CipStix} have given a more refined analysis and, consequently, explicit bounds on $p$. In particular, they have shown that for elliptic curves over $\Q$, $p$-divisibility holds for all $p > 7$ \cite[Theorem A]{CipStix}. Moreover, for a fixed elliptic curve over $\Q$, they show that $p$-divisibility can fail for at most one odd prime, and then for at most two quadratic twists of the given curve.

The main result of this paper is that, despite the mounting evidence, the answer to Cassels' question is in general no.

\begin{Theorem}
\label{MainTheorem2}
There exists an elliptic curve $A/\Q$ such that $\Sha(A) \not\subset 4\HH^1(A)$.
\end{Theorem}

\begin{Theorem}
\label{MainTheorem1}
For every prime $p$ there exists an abelian variety $A/\Q$ such that $\Sha(A) \not\subset p\HH^1(A)$.
\end{Theorem}

The key to these examples is the characterization of when $\Sha(A) \subset n\HH^1(A)$ given in Theorem \ref{ShaDivn} below. One consequence of this is an algorithm that, at least in principle, can determine whether $\Sha(A)[n] \subset n\HH^1(A)$ for any $n$. The example (see Theorem \ref{ECExample}) we give to prove Theorem \ref{MainTheorem2} is the curve of smallest conductor (it is $1025$) with this property. It was originally found using a crude implementation of the algorithm in {\tt Magma} \cite{magma}, though the verification we give here is different. The examples for Theorem \ref{MainTheorem1} are constructed using a slightly different approach. They come from Jacobians of cyclic covers of the projective line of genus $(p^3 - 3p + 2)/2$ defined over the $p$-th cyclotomic field. We obtain examples over $\Q$ (of dimension larger by a factor of $(p-1)$) using restriction of scalars.

\subsection*{Acknowledgments}
The author would like to thank Mirela \c Ciperiani and Jakob Stix for their comments, encouragement and useful discussions.

\section{Characterizing divisibility by $n$}
Let $\gal{k}$ denote the absolute Galois group of $k$. For a finite $\gal{k}$-module $M$ and a nonnegative integer $i$, let $\Sha^i(M)$ denote the kernel of the map $\HH^1(k,M) \to \prod\HH^1(k_v,M)$, the product running over all completions $k_v$ of $k$. When the base field is clear we will write $\HH^1(M)$ in place of $\HH^1(k,M)$.

It was Tate who showed that Cassels' pairing could be generalized to a pairing $\Sha(A) \times \Sha(\hat A) \to \Q/\Z$ whose left and right kernels are the divisible subgroups \cite{Tate} (Here $\hat A$ denotes the dual abelian variety). He notes that Cassels' methods suffice for the general case once one has recourse to his global duality theorems, and relates these to the pairing. The same tools allow one to characterize the divisibility of $\Sha(A)$ by $n$ in $\HH^1(A)$ using the pairing. While it is a simple observation, the fact that one can rule out divisibility in $\HH^1(A)$ this way seems to have gone unnoticed. 

\begin{Theorem}
\label{ShaDivn}
Let $A$ be an abelian variety over a number field and $n$ an integer. In order that $\Sha(A) \subset n\HH^1(A)$ it is necessary and sufficient that the image of the natural map $\Sha^1(\hat A[n]) \to \Sha(\hat A)$ is contained in the maximal divisible subgroup of $\Sha(\hat A)$.
\end{Theorem}

Since the maximal divisible subgroup of $\Sha(\hat A)$ is the (right) kernel of the pairing Theorem \ref{ShaDivn} is a consequence of the following.

\begin{Theorem}
\label{ShaDivphi}
Suppose $\phi:B \to A$ is an isogeny of abelian varieties over a number field with dual isogeny $\hat\phi :\hat B \to \hat A$, and let $T \in \Sha(A)$.  There exists $T' \in \HH^1(B)$ such that $\phi T' = T$ if and only if $T$ is orthogonal to the image of the map $\Sha^1(\hat A[\hat\phi]) \to \Sha(\hat A)$ induced by the inclusion $\hat A[\hat\phi] \subset \hat A$.
\end{Theorem}

\begin{proof}
Let $\delta : \HH^1(k,A) \to \HH^2(k,B[\phi])$ be the connecting homomorphism arising from the exact sequence $0 \to B[\phi] \to B \stackrel{\phi}\to A \to 0$. There exists a lift of $T$ as in the statement if and only if $\delta(T) = 0$. The kernels of $\phi$ and $\hat\phi$ are dual (as $\gal{k}$-modules), so Tate's duality theorem \cite[Theorem 3.1]{Tate} gives an isomorphism $\Sha^2(B[\phi]) \cong \Sha^1(\hat A[\hat\phi])^*$, where for an abelian group $G$, $G^* = \Hom(G,\Q/\Z)$. On the other hand, the Cassels--Tate pairing gives a map $\Sha(A) \to \Sha(\hat A)^*$. The `Weil pairing definition' of the Cassels--Tate pairing in \cite[Section 12.2]{PoonenStoll} can easily be extended to the case of an arbitrary isogeny in place of multiplication by an integer. Upon comparison of this definition with Tate's explicit description of his duality theorem as given in \cite[Section 3]{Tate} the reader may convince themselves that following diagram commutes.

\begin{align}
\label{diagram1}
\xymatrix{\Sha(A) \ar[r]^\delta \ar[d]^{} & \Sha^2(B[\phi]) \ar[d]^{\cong} \\
\Sha(\hat A)^* \ar[r] &\Sha^1(\hat A[\hat\phi])^*}
\end{align} Here the lower map is induced by the inclusion $\hat A[\hat\phi] \subset \hat A$.

Now suppose $T \in \Sha(A)$. Since $T$ is locally trivial, $\delta(T) \in \Sha^2(B[\phi])$. Since the vertical map on the right is an isomorphism, $\delta(T)$ is trivial if and only if its image in $\Sha^1(\hat A[\hat\phi])^*$ is. By commutativity, this image is the character
\[ \Sha^1(\hat A[\hat\phi]) \to \Sha(\hat A) \stackrel{ \langle T, \bullet\rangle}\To \Q/\Z\,. \] From this the statement is clear.
\end{proof}

\begin{Remark}
In regard to Theorem \ref{ShaDivphi}, it is well known that the stronger requirement that $T'$ lie in $\Sha(B)$ can be met if and only if $T$ is orthogonal to $\Sha(\hat A)[\hat\phi_*]$ (see e.g. \cite[Lemma 4.1]{CasselsIV}, \cite[Corollary to Theorem 1.2]{CasselsVIII}, \cite[Lemma 2.5]{CreutzANTS} or \cite[Lemma I.6.17]{ADT}). The proof of this fact, as well as the general proof that the kernels of the pairing are the maximal divisible subgroups essentially boils down to establishing the commutativity of diagram (\ref{diagram1}) above.
\end{Remark}

\begin{Remark}
In regard to Theorem \ref{ShaDivn}, the fact that the vanishing of $\Sha^1(\hat A[n])$ implies that the elements of $\Sha(A)$ are divisible by $n$ in $\HH^1(A)$ has been applied many times over. The results of Bashmakov, \c Ciperiani-Stix and Tate's result that $\Sha(E) \subset p\HH^1(E)$ all reduce to establishing that $\Sha^1(\hat A[n])$ is trivial under suitable circumstances. 
\end{Remark}

\begin{Remark}
From the Kummer sequence, \[ 0 \to \hat A(k) / n \hat A(k) \to \HH^1(\hat A[n]) \to \HH^1(\hat A)[n] \to 0\,,\] it is clear that the vanishing of $\Sha^1(\hat{A}[n])$ implies a local-global principle for divisibility by $n$ in $\hat A(k)$. This has been studied by Dvornicich and Zannier \cite{DZ}. If one is willing to assume that $\Sha(\hat A)$ contains no nontrivial elements, then Theorem \ref{ShaDivn} may be interpreted as saying that the nontrivial elements of $\Sha^1(\hat A[n])$ always pose a nontrivial obstruction to the local-global principle for divisibility by $n$, either in $\hat A(k)$ or in $\HH^1(A)$.
\end{Remark}

\section{The Example for Theorem \ref{MainTheorem2}}

\begin{Theorem}
\label{ECExample}
Let $E$ be the elliptic curve over $\Q$ with Weierstrass equation \[ E : y^2 = x(x+80)(x+205)\,.\] Then $\Sha(E) \not\subset 4\HH^1(E)$.
\end{Theorem}

\begin{Remark}
This is the smallest elliptic curve over $\Q$ (ordered by conductor) with this property. It is labelled {\tt 1025b2} in Cremona's Database \cite{CremonaDB}. 
\end{Remark}

\begin{proof}
We fix the ordered basis $P_1 = (0,0)$, $P_2 = (-80,0)$ for $E[2]$. For any $K/\Q$, this gives us an isomorphism $\HH^1(K,E[2]) \simeq \left(K^\times/K^{\times 2}\right)^2$ whose composition with the connecting homomorphism $\delta: E(K) \to \HH^1(K,E[2])$ is given explicitly by 
\[ P = (x,y) \longmapsto 
\begin{cases} 
(x,x+80) & \mbox{if } P \ne P_1,P_2 \\
(41,5) &  \mbox{if } P = P_1 \\
(-5,-1) & \mbox{if } P = P_2 \\
(1,1) & \mbox{if } P = 0 
\end{cases} \]
Let us consider the class $\xi \in \HH^1(\Q,E[2])$ represented by $(1,5)$. We claim that for every prime $p$, the image of $\xi$ in $\HH^1(\Q_p,E[2])$ lies in the image of $E[2]$ under the local connecting homomorphism $E(\Q_p) \to \HH^1(\Q_p,E[2])$. Assuming this we can complete the proof as follows. For any $K/\Q$ the exact sequence \[0 \to E[2] \to E[4] \to E[2] \to 0\] induces an exact sequence \[E[2] \stackrel{\delta}\to \HH^1(K,E[2]) \to \HH^1(K,E[4])\,.\] Thus our claim implies that the image of $\xi$ in $\HH^1(\Q,E[4])$ lies in $\Sha^1(E[4])$. On the other hand, $E$ has analytic rank $0$. This implies that $\Sha(E)$ is finite and, in particular, that it contains no nontrivial divisible elements. Also since $E(\Q) = E[2]$ and $\xi$ is clearly not contained in $\delta(E[2])$, the image of $\xi$ in $\Sha(E)$ is nontrivial. The result then follows from Theorem \ref{ShaDivn}.

It remains to establish the claim. Equivalently we must show that for every $p$, the image $\xi_p$ of $(1,5)$ in $(\Q_p/\Q_p^{\times 2})^2$ is represented by at least one of the pairs $(1,1), (41,5), (-5,-1)$, or $(-205,-5)$. If $5 \in \Q_p^{\times 2}$ or if $41 \in \Q_p^{\times 2}$, then $\xi_p$ is clearly represented by the first or by the second pair, respectively. This covers the case when $p \in \{ 2, 5, 41, \infty\}$. For any other $p$, all the entries lie in $\Z_p^\times$, which has exactly $2$ square classes. If $-5 \in \Q_p^{\times 2}$, then $-1$ and $5$ have the same square class, and we see that $\xi_p$ is represented by the third pair. If none of $5$, $-5$ or $41$ is a square in $\Q_p^\times$, then they all have the same square class and $-205 = (-5)\cdot 41 \in \Q_p^{\times 2}$, so $\xi_p$ is represented by the fourth pair. The cases considered are exhaustive, so the claim is proven.
\end{proof}

\begin{Remark} In the example $\Sha(E) \simeq \Z/2\Z\times \Z/2\Z$, and it is entirely possible to write down models for the corresponding torsors as double covers of $\PP^1$. The torsor corresponding to $\xi$ is given by
\begin{align*}
T_1 : y^2 &= (11x^2 - 67x + 31)(-x^2 - 3x - 1)\,.
\end{align*} The other two nontrivial elements are given by
\begin{align*}
T_2 : y^2 &= (11x^2 - 34x + 19)(x^2 + 6x + 4)\,,\text{ and }\\
T_3 : y^2 &= (11x^2 - 89x - 11)(-x^2 - x + 1)\,.
 \end{align*}
Since there are no elements of order $4$, $T_1$ pairs nontrivially with both $T_2$ and $T_3$. So by Theorem \ref{ShaDivphi} we see that $T_2,T_3 \not\in 4\HH^1(E)$. On the other hand, arguing as in the proof it is possible to show that neither $T_2$ nor $T_3$ lies in the image of $\Sha^1(E[4])$. Thus $T_1 \in 4\HH^1(E)$.
\end{Remark}

\section{The Examples for Theorem \ref{MainTheorem1}}
As mentioned in the introduction, our examples are Jacobians of cyclic covers of the projective line. We begin by fixing some notation. Let $X$ be a cyclic cover of $\PP^1$ of prime degree $p$ defined over a field $k$ of characteristic prime to $p$ containing the $p$-th roots of unity. We will assume that the covering is not ramified above $\infty \in \PP^1$ (this can always be arranged by a change of the coordinate on $\PP^1$ when $k$ is infinite). By Kummer theory, $X$ has a (possibly singular) model of the form $y^p = cf(x)$ where $c \in k^\times$ and $f \in k[x]$ has degree divisible by $p$ and leading coefficient $1$. Let $\Omega$ denote the set of ramification points. The $0$-cycles of degree $0$ supported on $\Omega$ generate a subgroup $J_\phi$ of the the $p$-torsion subgroup of the Jacobian $J = \Jac(X)$. There is an action of $\mu_p$ on $X/\PP^1$ given by $\zeta\cdot(x,y) = (x,\zeta y)$, which induces an inclusion of the cyclotomic ring $\Z[\mu_p]$ in $\End(J)$. The subgroup $J_\phi$ is the kernel of the isogeny $\phi = (1 - \zeta) \in \End(J)$ where $\zeta$ is a (hereupon fixed) primitive $p$-th root of unity.

For any $\omega \in \Omega$ the $1$-cocycle $\xi: \gal{k} \ni \sigma \mapsto \left(\sigma(\omega) - \omega\right) \in J_\phi$ represents a class in $\HH^1(J_\phi)$, which we will again denote by $\xi$. The class of this cocycle does not depend on the choice for $\omega$. One can easily determine if it is trivial with the following (see \cite[Lemma 11.2]{PoonenSchaefer}).

\begin{Lemma}
\label{xitrivial}
$\xi$ is trivial if and only if 
\begin{enumerate}
\item $f$ has a factor of degree prime to $p$, or
\item $p = 2$, $\deg(f) \equiv 2 \mod 4$ and $f(x)$ factors over some quadratic extension $K$ as $f = cg\bar{g}$ where $g, \bar{g} \in K[x]$ are $\Gal(K/k)$ conjugates.
\end{enumerate}
\end{Lemma}

The image of $\xi$ under the natural map $\HH^1(J_\phi) \to \HH^1(J)$ is represented by the torsor $\ALB^1_{X}$ paremeterizing $0$-cycles of degree $1$ on $X$. In the global situation, Lemma \ref{xitrivial} gives us a way to ensure that $\ALB^1_X$ lies in the image of $\Sha^1(J_\phi) \to \Sha(J)$. For example, if we choose $f$ so as to have a root in every completion this will be the case. The more involved task of arranging that $\ALB^1_X$ is not divisible in $\Sha(J)$ will be achieved with the following.

\begin{Lemma}
\label{NormCond}
Suppose $k$ is a number field and that $X$ has points everywhere locally. Let $L$ denote the $k$-algebra $k[x]/f$, and let $N_{L/k} : L \to k$ denote the norm. If $c \notin N_{L/k}(L^\times)k^{\times p}$, then $\ALB^1_X \not\in \phi\Sha(J)$.
\end{Lemma}

This follows directly from \cite[Theorem 4.6]{CreutzANTS}, and the details are to be found there. We simply summarize the principal train of thought. There is an explicit construction which, given $\alpha \in L^\times$ such that $cN(\alpha) \in k^{\times p}$, produces an unramified covering of $X$. The coverings produced this way are in fact torsors under $J_\phi$ (whose {\em type} in the sense of Colliot-Th\'el\`ene and Sansuc's theory of torsors under groups of multiplicative type is given by $J_\phi \hookrightarrow J(\bar{k})= \Pic^0(\bar{X}) \subset \Pic(\bar{X})$).  One can show that every such torsor with the property (call it {\bf P}) that the pullbacks of the ramification points are linearly equivalent to $k$-rational divisors arises in this way. Geometric class field theory gives a bijective correspondence between abelian $X$-torsors and abelian $\ALB^1_X$-torsors. One can show that any $\ALB^1_X$-torsor of type $J_\phi$ that has points everywhere locally corresponds to an $X$-torsor of type $J_\phi$ with property {\bf P} (assuming $X$ has points everywhere locally). On the other hand, a $k$-torsor $T$ under $J$ can be made into an $\ALB^1_X$-torsor under $J_\phi$ if and only if $\phi T = \ALB^1_X$ in $\HH^1(J)$. Thus $\ALB^1_X$ lies in $\phi \Sha(J)$ if and only there is some $\ALB^1_X$-torsor under $J_\phi$ that has points everywhere locally. As described above, this would imply that $c$ is a norm modulo $p$-th powers.

\begin{Remark}
In the case $p = 2$, Bruin and Stoll asked \cite[Question 7.2]{BruinStoll} if it were possible to find $X$ that is everywhere locally solvable for which $c$ is not a norm modulo squares. It was this question that originally motivated our construction of the examples in the following theorem.
\end{Remark}

\begin{Theorem}
\label{JacExamples}
Let $k = \Q(\zeta)$ where $\zeta$ is a primitive $p$-th root of unity. Let $p$ and $q$ be rational primes satisfying $q \equiv 1 \mod p^2$ (resp. modulo $8$ if $p = 2$), and set \[ f(x) = (x^p-\zeta)(x^p-q)(x^p-\zeta q) \cdots (x^p-\zeta^{p-1}q)\,.\] There are infinitely many $c \in k^\times$ such that $\Sha(J) \not\subset p \HH^1(J)$ where $J$ is the Jacobian of the cyclic cover of $\PP^1_k$ defined by $y^p = cf(x)$.
\end{Theorem}

\begin{proof}
The condition on $p$ and $q$ ensures that the first and second factors of $f(x)$ split into linear factors over $\Q_q$ and $\Q_p$, respectively. For primes of $k$ not lying above $p$ or $q$, Hensel's Lemma shows that the factorization of $f(x)$ can be determined by working over the residue field. By the pigeonhole principle at least one of the reductions of $\zeta, q, \zeta q,\dots, \zeta^{p-1}q$ must be a $p$-th power. Thus we conclude that $f(x)$ has a linear factor over every completion of $k$. By Lemma \ref{xitrivial} this implies that $\ALB^1_X$ lies in the image of $\Sha(J_\phi) \to \Sha(J)$ (for any $c \in k^\times$).

If $c \in k^\times$ is not a norm modulo $p$-th powers from $L = k[x]/f(x)$, then by Lemma \ref{NormCond} the class of $\ALB^1_X$ in $\Sha(J)$ is not divisible by $\phi$. As the induced pairing on $\Sha(J)/\phi\Sha(J)\times \Sha(J)[\phi]$ is nondegnerate there must exist some $T \in \Sha(J)[\phi]$ pairing nontrivially with $\ALB^1_X$. By Theorem \ref{ShaDivphi}, $T$ is not divisible by $\phi$. This implies that $T \notin p \Sha(J)$ since $\phi^{p-1}$ equals $p$ up to a unit in $\End(J)$. Thus the proof will be complete if we can show that $k^\times/N_{L/k}(L^\times)k^{\times p}$ is infinite.

We claim that if $\frak{r}$ is a prime of $k$ lying above a rational prime $r$ distinct from $p$ and $q$ such that $\frak{r}$ splits in the extension defined by $x^p - \zeta^iq$ for some $i \in \{1,\dots,p-1\}$, and is inert in the extension defined by $x^p-\zeta$, then $r \not\in N_{L/k}(L^\times)k^{\times p}$. Clearly the set of such primes has positive density. Since $[k:\Q] = p-1$ it will suffice to show that for any such $r$ we have $r \notin N_{L/\Q}(L^\times)\Q^{\times p}$. By way of contradiction, let us suppose $r$ is such a prime and that we can find $\alpha \in L^\times$ and $z \in \Q^\times$ such that $N_{L/\Q}(\alpha) = rz^p$.

Set $K_1 = k(\sqrt[p]{\zeta})$ and $K_2 = k(\sqrt[p]{q})$. For $i = 1,\dots,p-1$ the extensions of $k$ defined by $x^p - \zeta^iq$  all define the same extension of $\Q$, which we will denote by $K_3$. For a rational prime $s$ distinct from $p$ and $q$ let $\tau_i(s)$ be the greatest common divisor of the inertia degrees of the primes in $K_i$ above $s$. Then $\tau_1(s)$ is prime to $p$ if and only if $\zeta \in k_{\frak{s}}^{\times p}$ for every completion of $k$ at a prime $\frak{s}$ above $s$, and this occurs if and only if $s$ is a $p$-th power modulo $p^2$. Similarly $\tau_2(s)$ is prime to $p$ if and only if $q \in k_{\frak{s}}^{\times p}$ for every $\frak{s}$ above $s$. Finally $\tau_3(s)$ is prime to $p$ if and only if $\zeta q  \in k_{\frak{s}}^{\times p}$ for some $\frak{s}$ above $s$. Thus, when $s$ is not a square modulo $p^2$, at most one of $\tau_1(s)$, $\tau_2(s)$ and $\tau_3(s)$ can be prime to $p$. 

As a $\Q$-algebra, $L$ splits as $K_1\times K_2\times K_3 \times \dots \times K_3$, and the norm $N_{L/\Q}$ is simply the product of the norms of the individual factors. So without loss of generality we can assume the image of $\alpha$ under the splitting is given by $\alpha = (\alpha_1,\alpha_2,\alpha_3,1,\dots,1)$ with $\alpha_i \in K_i^{\times}$. The assumption on $r$ implies that $\tau_1(r)$ and $\tau_2(r)$ are divisible by $p$. So we see that the $r$-adic valuations of $N_{K_1/\Q}(\alpha_1)$ and $N_{K_2/\Q}(\alpha_2)$ are divisible by $p$. Thus the $r$-adic valuation of $N_{K_3/\Q}(\alpha_3)$ must be congruent to $1$ modulo $p$.

The completion of $K_3$ at the prime above $p$ is the $p^2$-cyclotomic extension of $\Q_p$ (since $q \in \Q_p^{\times p}$). Using higher ramification theory one gets that $N_{K_3/\Q}(\calO_{K_3}) \subset 1 + p^2\Z$ (or see \cite[Satz V.1.8]{Neukirch} for a direct proof). As $r$ is not a $p$-th power modulo $p^2$ (and in particular not in $1 + p^2\Z$), there must exist some rational prime $s$ distinct from $r$ and not in $1 + p^2\Z$ such that the $s$-adic valuation of $N_{K_3/\Q}(\alpha_3)$ is prime to $p$. This means that $\tau_3(s)$ is prime to $s$. Thus $\tau_1(s)$ and $\tau_2(s)$ are divisible by $p$. But then the $s$-adic valuation of $N_{L/\Q}(\alpha)$ is prime to $p$ which is a contradiction. This completes the proof
\end{proof}

To complete the proof of Theorem \ref{MainTheorem1} it suffices to note that for an extension $K/k$ and an abelian variety $A$ over $K$, restriction of scalars gives an isomorphism $\HH^1(K,A) \cong \HH^1(k,R_{K/k}(A))$. Thus the examples above also give rise to examples over $\Q$.


\begin{thebibliography}{19}

\frenchspacing
\renewcommand{\baselinestretch}{1}

\bibitem{Bashmakov64}
 M.I.~Ba\v smakov: {\em On the divisibility of principal homogeneous spaces over Abelian varieties}, Izv. Akad. Nauk SSSR Ser. Mat. {\bf 28} (1964) 661--664.
 
 \bibitem{Bashmakov72}
 M.I.~Ba\v smakov: {\em The cohomology of abelian varieties over a number field}, 
(Russian) Uspehi Mat. Nauk {\bf 27} (1972), 25--66;
(English Translation) Russian Math. Surv. {\bf 27} (1972) 25--70.

\bibitem{Bogomolov}
F.~Bogomolov: {\em Points of finite order on an abelian variety}, Izv. Akad. Nauk SSSR Ser. Mat., {\bf 44} (1980) 782--804.

\bibitem{magma}
  W.~Bosma, J.~Cannon and C.~Playoust: {\em The Magma algebra system. I. The user language}, J. Symbolic Comput. {\bf 24} (1997) 235--265.

\bibitem{BruinStoll} 
  N.~Bruin and M.~Stoll: {\em Two-cover descent on hyperelliptic curves}, Math. Comp. {\bf 78} (2009) 2347--2370.

\bibitem{CasselsIII}
 J.W.S.~Cassels: {\em Arithmetic on curves of genus 1. III. The Tate-\v Safarevi\v c and Selmer groups}, Proc. London Math. Soc. {\bf 12} (1962) 259--296.

\bibitem{CasselsIV}
 J.W.S.~Cassels: {\em Arithmetic on curves of genus 1. IV. Proof of the Hauptvermutung}, J. Reine Angew. Math. {\bf 211} (1962) 95--112.

\bibitem{CasselsVIII}
  J.W.S.~Cassels: \emph{Arithmetic on curves of genus 1. VIII. On conjectures
  of Birch and Swinnerton-Dyer}, J. Reine Angew. Math. {\bf 217} (1965) 180--189.

\bibitem{CipStix}
  M.~\c Ciperiani and J.~Stix: {\em Weil--Ch\^atelet divisible elements in Tate--Shafarevich groups II: On a question of Cassels}, 
  (preprint 2012).

\bibitem{CremonaDB}
  J.E.~Cremona: Elliptic curves database, available online at\\ {http://www.warwick.ac.uk/staff/J.E.Cremona/ftp/data}.

\bibitem{CreutzANTS}
  B.~Creutz: {\em Explicit descent in the Picard group of a cyclic cover of the projective line}, to appear in Proceedings of the Tenth Algorithmic Number Theory Symposium, {\tt arXiv:1204.5803}.

\bibitem{DZ}
  R.~Dvornicich and U.~Zannier: {\em Local-global divisibility of rational points in some commutative algebraic groups}, Bull. Soc. Math. France {\bf 129} (2001) 317--338.

\bibitem{ADT}
 J.S.~Milne: {\em Arithmetic Duality Theorems}, (second edition) Book Surge, LLC, 2006.

\bibitem{Neukirch}
 J.~Neukirch: Algebraische Zahlentheorie, Springer-Verlag Berlin Heidelberg, 1992.

\bibitem{PoonenSchaefer}
  B.~Poonen and E.F.~Schaefer: {\em Explicit descent for Jacobians of cyclic covers of the projective line}, J. Reine Angew. Math. {\bf 488} (1997) 141--188.
  
\bibitem{PoonenStoll}
  B.~Poonen and M.~Stoll: {\em The Cassels-Tate pairing for principally polarized abelian varieties}, Ann. of Math. {\bf 150} (1999) 1109--1149.   

\bibitem{Serre1}
 J.-P.~Serre: {\em Propri\'et\'es galoisiennes des points d'ordre fini des courbes elliptiques}, Invent. Math. {\bf 15} (1972) 259--331. 

\bibitem{Serre2}
J.-P.~Serre: {\em Points rationnels des courbes modulaires $X_0(N)$ [d'apr\`es Barry Mazur]}, S\'eminaire Bourbaki (1977/78), Exp. 511, Lecture Notes in Math. {\bf 710} (1979) 89--100.

\bibitem{Tate}
  J.~Tate: {\em Duality theorems in Galois cohomology over number fields}, Proc. Intern. Cong. Math. Stockholm (1967) 288--295.
\end{thebibliography}
\end{document}